\documentclass[12pt]{amsart}

\usepackage[top=2.5cm, bottom=2.5cm, left=3cm, right=2.8cm]{geometry}

\usepackage{amsthm}
\usepackage{enumitem}
\usepackage{graphicx}
\usepackage{xcolor}
\usepackage{fancyhdr}
\usepackage[colorlinks=true, linkcolor=blue!60!black, citecolor=blue!60!black, urlcolor=blue!60!black]{hyperref}
\usepackage{soul}

\theoremstyle{plain}
\newtheorem{maintheorem}{Theorem}

\newtheorem{theorem}{Theorem}[section]
\newtheorem{lemma}[theorem]{Lemma}
\newtheorem{proposition}[theorem]{Proposition}
\newtheorem{corollary}[theorem]{Corollary}

\theoremstyle{definition}
\newtheorem{definition}[theorem]{Definition}

\theoremstyle{remark}

\usepackage[colorlinks=true, linkcolor=blue!60!black, citecolor=blue!60!black, urlcolor=blue!60!black]{hyperref}

\usepackage{amsmath}
\DeclareMathOperator{\Ric}{Ric}
\DeclareMathOperator{\tr}{tr}
\DeclareMathOperator{\graph}{graph}

\begin{document}

\title{A Metric Entropy Formula for Magnetic Flows without Conjugate Points}

\author{Anthony J. García}
\email{anthonyjgg70@gmail.com}

\author{Guilherme Brandão Guglielmo}
\email{guilhermebrandaoguglielmo@gmail.com}

\subjclass{37D40 (primary), 53C15, 53C22, 37J25 (secondary).}

\keywords{Magnetic flows, Conjugate points, Green bundles, Entropy, Freire-Mañé formula, Focal points.}

\begin{abstract}
In this paper, we use the recently introduced notation and terminology for magnetic systems, which provide a natural framework for extending techniques and results from the theory of geodesic flows to magnetic flows. In this context, we prove a magnetic analogue of the classical Freire--Mañé theorem and obtain some applications. As a consequence, we show that for an energy level of a magnetic system without focal points, either the metric entropy is positive or the magnetic curvature is identically zero.

\end{abstract}

\maketitle 
\section{Introduction}
Magnetic flows provide a rich class of dynamical systems lying at the intersection of differential geometry, Hamiltonian dynamics, and ergodic theory. Over the last decades, many aspects of their dynamics have been extensively investigated. Among the various quantities used to measure dynamical complexity, entropy has played a particularly prominent role, and numerous formulas and bounds for both topological and metric entropy have been established; see \cite{grognet1999entropies,burns2002anosov,paternain1997first}.

Hyperbolic properties of magnetic flows have also attracted considerable attention. In particular, several authors have studied the existence of Anosov magnetic flows, invariant splittings, Lyapunov exponents, and their relations with geometric quantities such as curvature and the magnetic field; see \cite{gouda1997magnetic,wojtkowski2000magnetic,paternain1997regularity}.

Another central topic in the study of magnetic flows concerns the absence of conjugate points. As in the Riemannian setting, the absence of conjugate points imposes strong restrictions on the geometry and dynamics of the system and is closely related to rigidity phenomena, hyperbolicity, and the global behavior of magnetic geodesics. This condition has been investigated from several perspectives, leading to important developments in both geometry and dynamics. In particular, Bialy proved a magnetic analogue of the Hopf conjecture for conformally flat metrics on the torus, showing that a magnetic flow without conjugate points must be the geodesic flow of a flat metric \cite{bialy2000rigidity}.

Given a closed Riemannian manifold $(M,g)$ and a $2$-form $\sigma$ on $M$, let $\pi:TM\to M$ denote the canonical projection. Consider the twisted symplectic form
\[
\omega_\sigma=\omega_{\mathrm{can}}+\pi^*\sigma.
\]
The Hamiltonian $H(x,v)=\frac12\|v\|^2$ with respect to $\omega_\sigma$ generates a flow whose trajectories satisfy
\begin{equation}\label{mag eq}
    \nabla_{\dot\gamma}\dot\gamma=\Omega(\dot\gamma),
\end{equation}
where $\nabla$ denotes the Levi--Civita connection of $(M,g)$ and $\Omega$ is the bundle endomorphism determined by
\[
\sigma(u,v)=g(\Omega(u),v)
\]
for all $u,v\in TM$. The operator $\Omega$ is called the \emph{Lorentz force} associated with $\sigma$ and the curves satisfying \eqref{mag eq} \textit{magnetic geodesics.} When the $2$-form $\sigma$ is exact, the corresponding magnetic system can be regarded as a Tonelli Hamiltonian system on $T^*M$; see \cite{contreras1999global}.

The vector field generating \eqref{mag eq} is called the \emph{magnetic vector field} and is denoted by $X^{g,\sigma}$. Its flow $\Phi_t:TM\to TM$ is called the \emph{magnetic flow} associated with the pair $(g,\sigma)$. Since $H$ is preserved by the magnetic flow $\Phi_t$, it is natural to study the dynamics of $\Phi_t$ on each energy level. \\

\noindent \textbf{Convention.} In this article, we shall study the magentic flow $\Phi_t$  restricted  at level $\Sigma_s:=\{v\in TM:\|v\|=s\},$ for some $s>0$.\\

In the Riemannian setting, the Jacobi equation plays a fundamental role in describing the behavior of geodesics and in the study of manifolds without conjugate points. In the magnetic case, however, the presence of the Lorentz force introduces additional terms that make the variational equation considerably more complicated. An important step towards recovering the classical geometric picture was taken by Assenza \cite{assenza2024magnetic}, who introduced a twisted covariant derivative and a magnetic curvature operator. With these tools, the orthogonal component of a magnetic Jacobi field satisfies a second-order equation formally analogous to the classical Jacobi equation, making it possible to adapt several techniques from Riemannian geometry to the magnetic setting.

In particular, in the absence of conjugate points, Assenza, Reber, and Terek \cite{assenza2025magnetic} obtained a geometric description of the magnetic Green bundles $E^{\sigma,-},E^{\sigma;+}$in terms of magnetic Jacobi fields. More precisely, they showed that these bundles arise as limits of suitable Lagrangian subspaces determined by normal magnetic Jacobi fields, extending the classical constructions of Green \cite{green1958theorem}, Eberlein \cite{eberlein1973geodesic} for geodesic flows. This description allows one to associate with the Green bundles symmetric solutions $\mathcal U^+_{v}$ and $\mathcal U^-_{v}$ of the magnetic Riccati equation, and provides an effective tool for studying the geometry and dynamics of magnetic flows.

Our main result extends the classical formula of Freire and Mañé for geodesic flows \cite{freire1982entropy} to magnetic flows.
\begin{maintheorem}\label{teorema expoente}
Assume that the magnetic flow $\Phi_t:\Sigma_s\to\Sigma_s$ has no conjugate points, let $\mu$ be a $\Phi_t$-invariant probability measure, and denote by $\Psi_t=d\Phi_t$ the linearized magnetic flow. Then, for $\mu$-almost every $v\in\Sigma_s$, the limit
\[
\chi^\sigma(v):=\lim_{T\to+\infty}\frac1T\log\left|\det\!\left(\Psi_T\big|_{E^{\sigma,-}(v)}\right)\right|
\]
exists and satisfies
\[
\chi^\sigma(v)=\lim_{T\to+\infty}\frac1T\int_0^T\tr\bigl(\mathcal U^-_{\Phi_t(v)}(0)\bigr)\,dt.
\]

\end{maintheorem}

This formula was also obtained for Finsler geodesic flows by Foulon~\cite{foulon1992estimation}, and for convex Hamiltonian systems by Contreras and Iturriaga in \cite{contreras1999convex}, the latter including the class of exact magnetic flows. The main advantage of the approach presented here lies in its direct connection with the magnetic curvature tensor $M^\Omega_s$ (see Section~\ref{Geometric preliminaires}), introduced in \cite{assenza2024magnetic}, which allows us to exploit its geometric properties.

Theorem~\ref{teorema expoente} yields several interesting geometric bounds for the metric entropy of the magnetic flow with respect to invariant probability measures.

Let $(M,g,\sigma)$ be a magnetic system without conjugate points, and let
$\mu$ be a $\Phi_t$-invariant probability measure on $\Sigma_s$.
Define the number
\[
\mathfrak h_\mu^\sigma
:=\int_{\Sigma_s}\chi^\sigma\,d\mu.
\]

In general, Ruelle's inequality \cite{ruelle1978inequality} implies that for every invariant probability measure on $\Sigma_s$,
\[
h_\mu(\Phi)\le\mathfrak h_\mu^\sigma.
\]

It is well known that magnetic flows preserve the Liouville measure (see \cite{wojtkowski2000magnetic}, for more details). For invariant probability measures $\mu$ that are absolutely continuous
with respect to the Lebesgue measure on the energy level $\Sigma_s$,
Pesin's entropy formula \cite{pesin1977characteristic} implies that
\[
h_{\mu}(\Phi)=\int_{\Sigma_s}\chi^\sigma(v)\,d\mu.
\]

Combining Theorem~\ref{teorema expoente} with Birkhoff's Ergodic Theorem,
we obtain
\[
\int_{\Sigma_s}\chi^\sigma\,d\mu=\int_{\Sigma_s}\operatorname{tr}\bigl(\mathcal U^-_{v}(0)\bigr)\,d\mu.
\]

More precisely, we have proved the following. 
\begin{corollary}\label{Corolário Lebesgue}
Let $(M,g,\sigma)$ be a magnetic system without conjugate points on $\Sigma_s$, and let
$\mu$ be an absolutely continuous invariant measure on $\Sigma_s$.
Then
\[
h_{\mu}(\Phi)=\mathfrak h_{\mu}^{\sigma}=\int_{\Sigma_s}\tr\bigl(\mathcal U^-_v(0)\bigr)d\mu.
\]
\end{corollary}

The next result provides an upper bound for $\mathfrak h_{\mu}^{\sigma}$ in terms of the integral of the Ricci tensor.
\begin{corollary}\label{segundo corolario}
Let $(M,g,\sigma)$ be a magnetic system without conjugate points and
let $\mu$ be a $\Phi_t$-invariant probability measure on the magnetic
energy level $\Sigma_s$. Then
\[
\mathfrak h_\mu^\sigma\le(n-1)\left(-\int_{\Sigma_s}\Ric^\Omega_s(v)\,d\mu\right)^{1/2}.
\]

Moreover, if equality holds, then the magnetic curvature operator
$( M_s^\Omega)_v$ is a scalar multiple of the identity
for $\mu$-almost every $v$.
\end{corollary}
 As a consequence, magnetic flows with vanishing magnetic curvature at a given energy level have zero metric entropy with respect to the Liouville measure. A nontrivial example of such a system is provided by the horocycle flow, whose magnetic curvature vanishes at a distinguished energy level; see \cite{paternain2006magnetic}.

In their seminal paper \cite{freire1982entropy}, Freire and Mañé conjectured that every closed Riemannian manifold without conjugate points has positive metric entropy with respect to the Liouville measure, unless it is flat. To the best of our knowledge, the conjecture remains open in general, although some weaker versions have been established. For instance, Pesin proved the conjecture for Riemannian manifolds without focal points \cite{pesin1978equations}. Moreover, Knieper \cite{knieper1985mannigfaltigkeiten} proved that, on closed manifolds with continuous Green bundles, the topological entropy is positive if and only if the Liouville entropy is positive. In particular, combined with the work of Dinaburng \cite{dinaburg1971relations}, this implies the conjecture for closed surfaces whose Green bundles are continuous. In particular, combined with the work of Dinaburg \cite{dinaburg1971relations}, this implies the conjecture for closed surfaces whose Green bundles are continuous. However, continuity of the Green bundles is not automatic. Indeed, Burns and Ballmann \cite{ballmann1987surfaces} constructed examples of closed surfaces with discontinuous Green bundles.

In this article, we extend Pesin's result to magnetic flows without focal points. The two-dimensional case was recently obtained by Rebber in \cite{marshall2025marked}. The relevance of the no-focal-points condition is that it implies the positive semidefiniteness of the unstable Riccati operator.

\begin{maintheorem}\label{teorema rigidez}
Let $(M,g,\sigma)$ be a compact magnetic system without focal points, and let
$\mu_L$ be a Lebesgue measure on the energy level
$\Sigma_s$. If $(M,g,\sigma)$ is non-flat at level $s$ then then $h_{\mu_L}(\Phi)>0$.
\end{maintheorem}

\subsection*{Organization of the article}

In Section 2, we recall the geometric preliminaries of magnetic flows, including the magnetic curvature operator and the magnetic covariant derivative. In Section 3, we introduce the main tools used in the study of magnetic systems without conjugate points, namely the magnetic Green bundles and the divergence of radial Jacobi fields. Finally, in Section 4, we prove our main results.

\section{Geometric preliminaries}\label{Geometric preliminaires}

\subsection{Magnetic curvature}

Given $v\in T_1M$ and $w\in v^\perp:=\{w\in T_{\pi(v)}M: g(v,w)=0\}$, consider the bundle endomorphisms
\[
(R^\Omega_s)_v(w)=s^2R(w,v)v-s(\nabla_w\Omega)(v)+\frac{s}{2}\Big((\nabla_v\Omega)(w)-\langle (\nabla_v\Omega)(w),v\rangle v\Big),
\]
and
\[
(A^\Omega)_v(w)=\frac34\langle w,\Omega(v)\rangle \Omega(v)-\frac14\Omega^2(w)-\frac14\langle \Omega(w),\Omega(v)\rangle v .
\]

The magnetic curvature operator at level $s$ is defined by
\[
(M^\Omega_s)_v(w)=(R^\Omega_s)_v(w)+(A^\Omega)_v(w).
\]

This operator was introduced by Assenza in \cite{assenza2024magnetic}. The magnetic sectional curvature is given by
\[
(\sec^\Omega_s)_v(w)=\big\langle (M^\Omega_s)_v(w),w\big\rangle ,
\qquad w\in v^\perp,\ \|w\|=1,
\]
and the magnetic Ricci curvature by
\[
\Ric^\Omega_s(v):=\frac1{n-1}\tr\bigl(( M_s^\Omega)_v\bigr).
\]

The magnetic sectional curvature generalizes the magnetic curvature previously introduced for magnetic flows on surfaces; see \cite{paternain1993anosov}.

\subsection{Magnetic covariant derivative}
The anisotropic Lorentz force $\widetilde \Omega$ is the skew-adjoint operator given by \begin{align}
    \widetilde \Omega_v(w)=\frac{1}{2}(\langle\Omega(w),v\rangle v+\langle w,v\rangle \Omega(v)+\Omega(w))
\end{align}

This operator induces a derivation along $\gamma$ given by
\begin{align}
   \widetilde{\mathcal D}V   =   \frac{DV}{dt}   -
   \widetilde{\Omega}_{\frac{1}{s}\dot\gamma}(V).
\end{align}
Since $\widetilde{\Omega}_{\frac{1}{s}\dot\gamma}$ is skew-adjoint, the derivation $\widetilde{\mathcal D}$ is compatible with the metric $g$. Moreover, the parallel transport induced by $\widetilde{\mathcal D}$ is an isometry with respect to $g$ (see Lemma~11 in \cite{assenza2024magnetic}). We denote by
\[
P_t:T_{\gamma(0)}M\to T_{\gamma(t)}M
\]
the $\widetilde{\mathcal D}$-parallel transport along $\gamma$.

\subsection{Sasaki twisted decomposition}

For every $v \in TM$ there exists a natural isomorphism
\[
j_v:T_vTM \to T_{\pi(v)}M \oplus T_{\pi(v)}M,\qquad\xi \mapsto (d_v\pi(\xi),K_v(\xi)),
\]
where $K_v:T_vTM \to T_{\pi(v)}M$ is the connection map induced by the connection $\nabla$ (see \cite{paternain1999geodesic}, for details). In this way, we can identify $T_vTM$ with $H_v \oplus V_v$ via $j_v$, where $H_v=\ker(K_v)$ and $V_v=\ker(d_v\pi)$, are called the \textit{horizontal} and \textit{vertical} spaces, respectively.

The \textit{twisted sympletic} for $\omega^\sigma$ on $TM$ is given by
\begin{equation}\label{eq:twisted_symplectic}
\omega_\sigma(\xi,\eta)=\langle \pi_*\xi, K(\eta)\rangle-
\langle K(\xi), \pi_*\eta\rangle-\sigma(\pi_*\xi,\pi_*\eta).
\end{equation}

Consider the \textit{twisted connector } $K^\sigma:TTM\setminus\{0\}\to TM$ given by 
\[
K^\sigma(\xi)=K(\xi)-\widetilde \Omega(\pi_*(\xi)).
\]

The twisted horizontal space is defined by
$H_v^\sigma=\ker\!\left(K_v^\sigma|_{T_v\Sigma_s}\right).$
We also define
$\Theta_v=\{\xi\in T_v\Sigma_s:\pi_*(\xi)\in\mathbb R\,\Omega(v),\,
K(\xi)=0\}$, and $\widehat{ H}_v^\sigma= H_v^\sigma\oplus\Theta_v$.
We denote by $\mathcal Q_v:=T_v\Sigma_s/\mathbb R X^{g,\sigma}(v)$.
The space $\mathcal Q_v$ has dimension $2(n-1)$ and plays the role of the normal space in the geodesic setting. In particular, one can show that $\omega^\sigma$ induces a nondegenerate symplectic form on $\mathcal Q$. Since $X^{g,\sigma}(v)\notin \mathcal V_v=V_v|_{\Sigma_s}$, we identify $\mathcal Q_v\simeq(\widehat{ H}_v^\sigma / \mathbb R X^{g,\sigma}(v))\oplus\mathcal V_v$. These facts were established by Assenza, Reber, and Terek in \cite{assenza2025magnetic}.
For simplicity, we write $\mathcal H_v^\sigma=\widehat{ H}_v^\sigma / \mathbb R X^{g,\sigma}(v)$.

 As in the geodesic setting, we may define an isomorphism
\begin{equation}\label{correspondencia magnetica}
j_v^\sigma:\mathcal Q_v \to v^\perp\oplus v^\perp,
\qquad
\xi \mapsto \bigl([d_v\pi(\xi)]^\perp,K_v^\sigma(\xi)\bigr).
\end{equation}

 The 2-form $\omega^\sigma$ is sympletic restricted to $\mathcal Q_v$ and  $\Psi_t$ is a symplectomorphis on $(\mathcal{Q},\omega^\sigma|_{\mathcal{Q}}).$

\subsection{Magnetic Jacobi fields } Given a $(g,\sigma)$-geodesic $\gamma$, the $(g,\sigma)$-\textit{Jacobi} equation associated to $\gamma_v$ is given by 
\begin{equation}
    \ddot J+R(J,\dot\gamma)\dot\gamma=(\nabla_J\Omega)(\dot\gamma)+\Omega(\dot J)
\end{equation}
where $\cdot$ is the covariant derivative associated to the Levi Civita connection $\nabla$ and $R$ is the Riemannian tensor of curvature. We said that a $(g,\sigma)$-Jacobi field is \textit{normal} if $\langle \dot J(t),\dot \gamma(t)\rangle=0$ for all $t \in \mathbb R.$

Let $J$ be a $(g,\sigma)$-Jacobi field along $\gamma$. Writing
\[
J=f\dot\gamma+J^\perp,\qquad \langle J^\perp,\dot\gamma\rangle=0,
\]
we have
\[
f=\langle J,\dot\gamma\rangle.
\]

Differentiating and using that $\gamma$ is a magnetic geodesic, namely
\[
\nabla_{\dot\gamma}\dot\gamma=\Omega(\dot\gamma),
\]
it follows that
\[
f'=\langle \dot J,\dot\gamma\rangle+\langle J,\Omega(\dot\gamma)\rangle.
\]

In particular, if $J$ is normal, i.e.
\[
\langle \dot J,\dot\gamma\rangle=0,
\]
then
\[
f'=\langle J,\Omega(\dot\gamma)\rangle=\langle J^\perp,\Omega(\dot\gamma)\rangle.
\]

Hence, the tangential component of a normal Jacobi field is completely determined by its normal component. Using the correspondence  (\ref{correspondencia magnetica}), we recover the classical result for geodesic flows:
\begin{proposition}\label{correspoendencia diferencial -campos de Jacobi}
    For every $\xi\in\mathcal{Q}_v$, there exist a linear correspondence such that 
    \[
    \Psi_t(\xi)\simeq(J^\perp_\xi(t),\widetilde DJ_\xi(t))
    \]
    where $J_\xi$ is the unique normal $(g,\sigma)$-Jacobi field such that $J_\xi(0)=[\pi_*\xi]^\perp$ and $\widetilde DJ_\xi(0)=K^\sigma(\xi).$
\end{proposition}
\begin{proof}
    It is known that $j\,(\Psi_t(\xi))=(J_\xi(t),DJ_\xi(t))$,   where $J_\xi(0)=\pi_*(\xi)$ and $DJ_\xi(0)=K(\xi)$ (see Lemma 5.3 in \cite{gouda1997magnetic}).  Using correspondence, we obtain \eqref{correspondencia magnetica} $$
        j^\sigma(\Psi_t(\xi))=(J_\xi(t)^\perp,DJ_\xi(t)-\Omega(J_\xi(t)))
        =(J_\xi(t)^\perp,\widetilde DJ_\xi(t)).
    $$
\end{proof}
In \cite{assenza2025magnetic}, it was proved that if $J$ is a Jacobi field along $\gamma$, then its orthogonal component $J^\perp$ satisfies the reduced Jacobi equation (also called the orthogonal Jacobi equation)
\begin{equation}\label{normal Jacobi eq}
    \widetilde{\mathcal{D}}^2Y+M^\Omega_s(J^{\perp})=0.
\end{equation}
In fact, the argument presented in \cite{assenza2025magnetic} can be reversed, as the following proposition shows.

\begin{proposition} \label{normal jacobi eq reciproca}
Let $Y$ be a normal vector field along $\gamma$ satisfying
\[
\widetilde{\mathcal{D}}^2J^{\perp}+M^\Omega_s(J^{\perp})=0.
\]
Then the vector field
\[
J=f\dot\gamma+J^{\perp},
\]
is a normal Jacobi field along $\gamma$,
where
\[
f'(t)=\frac{1}{s}\left\langle J^{\perp}(t),\Omega(\dot\gamma(t))\right\rangle.
\]

\end{proposition}
\begin{proof}
   Define $\mathfrak J(J)=\ddot J+R(\dot\gamma,J)\dot\gamma-(\nabla_J\Omega)(\dot\gamma)-\Omega(\dot J)$. A direct calculation shows that
\[
\mathfrak J(J)=f''\dot\gamma+f'\Omega(\dot\gamma)+\mathfrak J(J^{\perp}),
\]
and $\langle\mathfrak J(J^{\perp}),\dot\gamma\rangle=-f''$. Therefore, $\langle\mathfrak J(J),\dot\gamma\rangle=0$. On the other hand, following Proposition~3.1 in \cite{assenza2025magnetic}, we see that
\[
(\mathfrak J(J))^{\perp}=\widetilde{\mathcal D}^2J^{\perp}+M^\Omega_s(J^{\perp})=0.
\]
it follows that $\mathfrak J(J)=0$. Therefore, $J$ is a $(g,\sigma)$-Jacobi field.
\end{proof}
\section{Magnetic systems without Conjugate points} Given a magnetic geodesic $\gamma$, we say that two points
$\gamma(a)$ and $\gamma(b)$ are \textit{conjugate} if there exists a nontrivial
normal Jacobi field $J$ along $\gamma$ such that $J(a)=0$ and $J(b)$ is parallel
to $\dot\gamma(b)$. We say that a magnetic system $(M,g,\sigma)$ has no conjugate
points on the energy level $\Sigma_s$ if every magnetic geodesic on $\Sigma_s$
has no conjugate points.

It is well known that magnetic systems with nonpositive magnetic sectional
curvature have no conjugate points. One of the most important characteristics of
flows without conjugate points is the existence of the fiber bundles called
\textit{Green bundles}, which we describe in this section. The existence of such
bundles was proved by Green for geodesic flows \cite{green1958theorem}, by Foulon
\cite{foulon1992estimation} for Finsler metrics, by Contreras and Iturriaga for
Hamiltonian flows, and more recently by Assenza, Reber, and Terek
\cite{assenza2025magnetic} for magnetic flows.

For every $v \in \Sigma_s$ and $T\in\mathbb R$, we denote
$E_v(T):=\Psi_{-T}(\mathcal V_{\Phi_T(v)})$. Since $\mathcal V_v$ is $\omega^\sigma$-Lagrangian and $\Psi_t$ preserves $\omega^\sigma$, then
$E_v(T)$ is also Lagrangian. Moreover, the assumption of no conjugate points implies that $E_v(T)\cap \mathcal V_v=\{0\}$. Therefore, $E_v(T)$ is the graph of a symmetric operator $S_v(T):\mathcal H_v^\sigma \to \mathcal V_v$.

\begin{proposition}\label{Green properties}\cite{assenza2025magnetic}
For each $v \in \Sigma_s$, there exist two invariant, $n-1$ dimensional subspaces $E^{\sigma,+}_v$ and $E^{\sigma,-}_v$ of $\mathcal{Q}_v$, defined by
\begin{align*}
    E^{\sigma,+}_v &= \lim_{t \to +\infty} \Psi_{-t}(\mathcal{V}_{\phi_t(v)}), \\
     E^{\sigma,-}_v &= \lim_{t \to +\infty} \Psi_{t}(\mathcal{V}_{\phi_{-t}(v)}),
\end{align*}
where $\mathcal{V}_v=V_v\cap\mathcal{Q}_v$ is the vertical subspace at $v$.The distributions $E^{\sigma,+}_v$ and $E^{\sigma,-}_v$, called the \emph{stable} and \emph{unstable Green bundles}, respectively, satisfy the following properties:

\begin{enumerate}
    \item They are measurable, transverse to the vertical subbundle $\mathcal{V}_v$, and transverse to the geodesic vector field.
    \item There exist linear operators $\mathcal  S^+_v : {\mathcal H}_v^\sigma\to \mathcal V_v$, $\mathcal S^-_v :{\mathcal H}_v^\sigma \to \mathcal V_v$, such that $ E^{\sigma,+}_v $, $ E^{\sigma,-}_v $ are, respectively, the graphs of $\mathcal  S^+_v$, $\mathcal  S^-_v$. Such operators give rise to solutions of a well-known Riccati equation.
\end{enumerate}
\end{proposition}

\begin{definition} Given a magnetic geodesic $\gamma_v$, an application $\mathcal J(t):v^\perp \to v ^\perp$ is called a $(g,\sigma)$--\textit{Jacobi Tensor} if it satisfies the equation:
\begin{equation}\label{Jacobi eq}
    \ddot {\mathcal J}+\mathcal M  \mathcal J=0
\end{equation}
    where $\mathcal M=P_{-t}\circ M_s^\Omega \circ P_t$ and $\dot {\mathcal J}w=\frac{d}{dt}(\mathcal J w).$
\end{definition}

\subsection{Ricatti equation}
Whenever a solution $\mathcal J$ of \eqref{Jacobi eq} is invertible on an interval $(a,b)$, the matrix
$U(t)=\dot {\mathcal J}(t){\mathcal J}(t)^{-1}$ satisfies the $(g,\sigma)$-\textit{Riccati equation}
 \begin{equation}\label{Riccati eq}
     \dot U(t)+U^2(t)+\mathcal M(t)=0,
 \end{equation}
for $t\in (a,b).$ In fact,
\begin{align*}
    \dot U(t)=&\ddot {\mathcal J}(t){\mathcal J}^{-1}(t)-\dot{\mathcal J}(t)J^{-1}(t)\dot {\mathcal J}(t)J^{-1}(t)\\
    =&-\mathcal M(t){\mathcal J}^{-1}(t){\mathcal J}(t)+U^2(t)\\
    =&-\mathcal M(t)-U^2(t).
\end{align*}
Given two solutions ${\mathcal J}_1,{\mathcal J}_2$ of the matricial equation \eqref{Jacobi eq}, the \textit{Wronskian} is defined by
\[
W({\mathcal J}_1(t),{\mathcal J}_2(t)):=(\dot{\mathcal J}_1(t))^T{\mathcal J}_2(t)-{\mathcal J}_1(t)^T\dot{\mathcal J}_2(t).
\]
Since $\mathcal M(t)$ is symmetric, a standard argument implies that $W({\mathcal J}_1(t),{\mathcal J}_2(t))$ is constant.

Equation~\eqref{Jacobi eq} allows us to recover several classical results due to Green \cite{green1958theorem}, Eberlein \cite{eberlein1973geodesic}, and others. As a consequence, we obtain the following result.

\begin{proposition}\label{Riccati limitado} Let $\mathcal J$ a Jacobi Tensor and $U(t):=\dot{\mathcal J}(t)\mathcal J(t)^{-1}$. Suppose that $U(t)$ is defined for every $t>0$. Then:

\begin{enumerate}

\item The operator $U(t)$ is symmetric for every $t>0$ if and only if the Wronskian 
\[
W({\mathcal J}(t),{\mathcal J}(t))=0
\]

\item If $U(t)$ is symmetric, then $U$ is uniformly bounded in $[1,\infty)$.

\end{enumerate}
\end{proposition}

\subsection{Stable and unstable Jacobi tensors} The isomorphism \eqref{correspondencia magnetica} allows us to identify
\[
\mathcal H_v^\sigma \simeq v^\perp
\qquad\text{and}\qquad
\mathcal V_v \simeq v^\perp.
\]

Recall that $E_v^{\sigma,+}=\graph(\mathcal S_v^+)$ and let
\[
\xi_w\in E_v^{\sigma,+}
\qquad\text{with}\qquad
\xi_w=(w,\mathcal S_v^+w),
\quad w\in v^\perp.
\]
Consider the normal magnetic Jacobi field $J_w$ along $\gamma_v$, with initial data
\[
J_w^\perp(0)=w,
\qquad
\widetilde D J_w(0)=\mathcal S_v^+w.
\]
We then define
\[
Y_v^+(t)w:=P_{-t}\bigl(J_w^\perp(t)\bigr).
\]

Since $P_t$ is $\widetilde D$-parallel,
\[
\dot Y_v^+(t)w
=
P_{-t}\bigl(\widetilde D J_w^\perp(t)\bigr).
\]
The vector $\Psi_t(\xi_w)$ belongs to $E_{\Phi_t(v)}^{\sigma,+}$; therefore, by Proposition~\ref{correspoendencia diferencial -campos de Jacobi},
\[
\bigl(J_w^\perp(t),\,\widetilde D J_w(t)\bigr)
\in
\graph\bigl(\mathcal S_{\Phi_t(v)}^+\bigr),
\]
and consequently
\[
\widetilde D J_w(t)
=
\mathcal S_{\Phi_t(v)}^+\bigl(J_w^\perp(t)\bigr).
\]
Applying $P_{-t}$ to this identity yields
\[
\dot Y_v^+(t)w=(P_{-t}\circ \mathcal S_{\Phi_t(v)}^+\circ P_t)\bigl(Y_v^+(t)w\bigr).
\]
Since this holds for every $w\in v^\perp$, we obtain the matrix relation
\[
\dot Y_v^+(t) = (P_{-t}\circ \mathcal S_{\Phi_t(v)}^+\circ P_t)\,Y_v^+(t).
\]
Since the flow has no conjugate points, every nonzero vector $
(w,S_v^+w)\in E^{\sigma,+}(v)$ determines a magnetic Jacobi field $J_w^+$. Therefore, $Y_v(t)w=P_{-t}\bigl((J_w^+)^\perp(t)\bigr)\neq 0$ if $w\neq 0$. Hence $Y^+_v(t)$ is injective. Since $Y^+_v(t):v^\perp\to v^\perp$ is a linear map between vector spaces of the same finite dimension, it follows that $Y^+_v(t)$ is invertible, 
\[
\dot Y_v^+(t)(Y^+_v(t))^{-1}=(P_{-t}\circ \mathcal S_{\Phi_t(v)}^+\circ P_t)
\]
In particular, 
\[
Y^+_v(0)=Id \,\,\mbox{and}\,\,\, \dot Y^+_v(0)=\mathcal S^+_v.
\]

The same argument can be used to construct an operator $Y_v^-$ associated with $E_v^{\sigma,-}$.

\begin{definition}
The operators $Y_v^+$ and $Y_v^-$ are called the \textit{stable and unstable Jacobi--Green operators}, respectively. The operators $\mathcal U_v^+:=\dot Y_v^+(Y_v^+)^{-1}$ and $\mathcal U_v^-:=\dot Y_v^-(Y_v^-)^{-1}$ are called the \textit{stable and unstable Riccati operators}, respectively. 
\end{definition}
The correspondences $v\mapsto \mathcal U_v^\pm(0)$ are measurable. Moreover, they satisfy 
\[
\mathcal U_v^\pm(t)=P_{-t}\circ \mathcal U_{\Phi_t(v)}^\pm(0)\circ P_t.
\]

Note that, as an immediate consequence of Proposition \ref{Riccati limitado}, $\mathcal U_v^+$ and $\mathcal U_v^-$ are symmetric and uniformly bounded for all $t\in\mathbb{R}$.

\subsection{Divergence of magnetic Jacobi rays}
\begin{itemize}[label=$\bullet$]
\item Let $A$ be the operator such that $A(t):v^\perp\to v^\perp$ and $A(t)w=P_{-t}(J^\perp_w(t))$ where $J_w(0)=0$ and $\widetilde DJ_w(0)=w$. Then $A$ is a solution of \eqref{Jacobi eq} satisfying $A(0)=0$ and $\dot A(0)=Id$. Since $(M,g,\sigma)$ has no conjugate points, $A$ is invertible on $(0,+\infty)$. Indeed, if $A(T)w=0$ for some $T>0$ and $w\neq0$, then the vector field $J_o(t)=P_t(A(t)w)$ would satisfy $J_o(0)=0$ and $J_o(T)=0$. Then, the vector field $J=f\dot\gamma+J_o$, where $f'=\frac{1}{s}\langle J_o,\Omega(\dot \gamma)\rangle$ and $f(0)=0$, would be a normal $(g, \sigma)$-Jacobi field by Proposition \ref{normal jacobi eq reciproca} with conjugate points, which is a contradiction. Therefore, $A(t)$ is invertible for all $t>0$.

\item Let $B(t):v^\perp\to v^\perp$ be the operator defined by
$B(t)w=P_{-t}(J_w^\perp(t))$, where $J_w$ is the solution satisfying
$J_w(0)=w$ and $\widetilde D J_w(0)=0$. Then $B$ is a solution of
\eqref{Jacobi eq} and satisfies
$B(0)=Id$ and $\dot B(0)=0$.

\item Let $Y_{T}(t):v^\perp \to v^\perp$ be such that $Y_{T}(t)=B(t)-A(t)A(T)^{-1}B(T)$. Then $Y_{T}(0)=Id$ and $Y_{T}(T)=0$. Moreover, $Y_{T}$ satisfies \eqref{Jacobi eq} and $Y_{T}(t)w=P_t(J_{w,T}(t)^\perp)$. Where $J_{w,T}$ is the Jacobi field satisfying $J_{w,T}(0)=w$ and $(J_{w,T}(T))^\perp=0.$ Then $J_{w,T}\in E_v^\sigma(T)$ via the correspondence \eqref{correspondencia magnetica} and hence $\widetilde DJ_{w,T}(0)=\mathcal S_v(T)w$. Therefore, $\dot Y_T(0)=\widetilde DJ_{w,T}(0)^\perp=\widetilde DJ_{w,T}(0)=\mathcal{S}_v(T)w.$ Then Lemma 4.4 in \cite{assenza2025magnetic} implies that $Y_T(0)\to Y^+(0)$ as $T\to \infty$ and as $Y_T\to Y$ as $T\to \infty.$
\end{itemize}

\begin{proposition}\label{Jacobi divergencia} For all $R > 0$ there exists $T > 0$ such that $|A(t)x| > R|x|$ for all $|t |> T$ and all $x \in T_v M \setminus \{0\}$.
    
\end{proposition}
\begin{proof}
    A standard Wronskian argument shows that
    \[
    Y_{v,t}(s)=A(s)\int_s^t A^{-1}(u)(A^{-1}(u))^* \, du,
    \]
 for all $0<s<t$. Denote $D(s):=\int_s^{+\infty} A^{-1}(u)(A^{-1}(u))^* \, du$, $W(s):= \dot A(s)(A^{-1}(s))$ and $\mathcal U^+(s)=\dot Y^+(s)(Y^+)^{-1}(s)$, for sure, $W$ and $U$ satisfies the Riccati equation (\ref{Riccati eq}). Moreover, one can verify that
\[
\mathcal U^+(s)-W(s)=(A^{-1}(s))^TD^{-1}(s)A^{-1}(s).
\]

Since $\mathcal U^+$ and $W$ are uniformly bounded by Proposition \ref{Riccati limitado} in $[1,+\infty)$, there exists a constant $k>0$ such that for every $x \in \mathbb R$ with $\|x\|=1$,
\[
\left|\left\langle (A^{-1}(s))^TD^{-1}(s)A^{-1}(s)x,x\right\rangle\right|=\left|\left\langle D^{-1}(s)A^{-1}(s)x,A^{-1}(s)x\right\rangle\right|<k.
\]

Let $\lambda_{\max}(s)$ denote the eigenvalue of $D(s)$ with largest modulus. Since $D(s)$ is positive definite, we obtain
\[
\left\langle D^{-1}(s)A^{-1}(s)x,A^{-1}(s)x\right\rangle \geq
\lambda_{\max}(s)^{-1}\|A^{-1}(s)x\|^2.
\]
Hence,
\[
\|A^{-1}(s)x\|^2<k\,\lambda_{\max}(s).
\]

Because $D(s)$ is symmetric and positive definite,
\[
\lambda_{\max}(s)=\|D(s)\|,
\]
and therefore
\[
\|A^{-1}(s)x\|^2 < k\|D(s)\|.
\]

Finally, since $D(s)\to 0$ as $s\to +\infty$, it follows that
\[
\|A^{-1}(s)x\|\to 0\qquad \text{as } s\to +\infty.
\]
Then $\|A(s)x\|\to +\infty$ as $s\to +\infty$. Replacing $S^+$ by $S^-$ in the above construction, we obtain the divergence formula in backward time.
\end{proof}

\section{Proof of main results}

The proof of the Freire-Mañé magnetic formula is based on Oseledets' theorem \cite{oseledec1968multiplicative}. Applied to $\Psi_t$ on $\Sigma_s$, it provides a measurable decomposition of $Q_v$ into a complete measure set. This decomposition will be used to compare the unstable magnetic Green bundle with Oseledets subspaces.
Recall that, for every $\Phi_t$-invariant probability measure $\mu$, Oseledets' theorem gives a full $\mu$-measure set $\mathcal O\subset\Sigma_s$ such that for every $v\in\mathcal O$ there exists a measurable splitting
\[
Q_v = E^s(v)\oplus E^c(v)\oplus E^u(v),
\]
where the Lyapunov exponents on $E^s(v)$ are negative, on $E^c(v)$ are zero, and on $E^u(v)$ are positive. Our goal is to demonstrate that the unstable Green bundle is between the unstable and central Oseledets subspaces.

\begin{proposition}\label{osedelecs}
Let $\mu$ be an invariant probability measure for $\Phi_t$, and let
\[
Q_v=E^s(v)\oplus E^c(v)\oplus E^u(v)
\]
be the Oseledec decomposition of $\Psi_t$ on a full measure set.

Then, for $\mu$-almost every $v\in\Sigma_s$,
\[
E^u(v)\subset E^{\sigma,-}(v)\subset E^u(v)\oplus E^c(v).
\]
In particular, the limit
\[
\lim_{T\to+\infty}\frac1T\log\left|\det\!\left(\Psi_T\big|_{E^{\sigma,-}(v)}\right)\right|
\]
exists for $\mu$-almost every $v$.
\end{proposition}

\begin{proof} 

Fix $v\in \mathcal{O}$ and write
\[
F(v):=E_v^{\sigma,-}=
\lim_{t\to+\infty}\Psi_t\bigl(V_{\Phi_{-t}(v)}\bigr).
\]

The proof is based on the argument of Freire and Mañé for geodesic flows.\\

\noindent \textbf{Step 1: proof of the inclusion $F(v)\subset E^u(v)\oplus E^c(v)$.} We claim that the vertical subspace $\mathcal V_v$ is
transverse to $E^s(v)$. Indeed, let $\xi\in \mathcal V_v\setminus\{0\}$,
and let $J_\xi$ be the corresponding normal magnetic Jacobi field along
the geodesic $\gamma_w$ given by Proposition \ref{correspoendencia diferencial -campos de Jacobi}. By Proposition~\ref{Jacobi divergencia}, given $R>1$, there exist $T>0$ such that the
normal component of $J_\xi$ satisfies:
\[
\|J_\xi^\perp(t)\|> R
\qquad \text{for all } t>T.
\]
Hence
\[
\liminf_{t \to \infty}\frac{1}{t}\log \|\Psi_t(\xi)\|\geq\liminf_{t\to+\infty}\frac1t\log \|J_\eta^\perp(t)\|\ge 0.
\]
On the other hand, if $\xi\in E^s(v)$, then by definition of the
stable subspace 
\[
\lim_{t\to+\infty}\frac1t\log \|\Psi_t(\xi)\|
<0,
\]
which is impossible. Therefore
\[
\mathcal V_v\cap E^s(v)=\{0\}.
\]

Since
\[
Q_v=E^s(v)\oplus E^c(v)\oplus E^u(v),
\]
it follows that that there exist $R(v)\subseteq E^c(v)\oplus E^u(v)$ such that $\mathcal V_v$ is the graph of a unique linear operator
\[
 L_v:R(v)\to E^s(v),
\]
that is,
\[
\mathcal V_v=\graph(L_v).
\]
Fix $\varepsilon>0$ small. By standard measure-theoretic arguments, we can find a compact set $K\subset \mathcal O$ with $\mu(K)>1-\varepsilon$ such that, for every $v \in K$,

\begin{itemize}
    \item $\sup_{\xi\in R(v)}\|L_v(\xi)\|=L_M<\infty$,
    \item There exist constants
\[
T>0,\qquad 0<\eta<1<\lambda<\eta^{-1},
\]
such that, for all $t>T$,
\[
\|\Psi_t(\xi)\|<\eta^t\|\xi\|
\quad\text{for all }\xi\in E^s(v),
\qquad
\|\Psi_{-t}(\zeta)\|<\lambda^t\|\zeta\|
\quad\text{for all }\zeta\in R(v).
\]
\end{itemize}

Fix $v\in K$. By the Poincaré Recurrence Theorem, there exists a sequence $(t_n)$ with $t_n\to\infty$ such that
$
\Phi_{-t_n}(v)\in K
$
for every $n$. Consider the transported subspace
\[
V_{-t_n}(v):=\Psi_t\bigl(V_{\Phi_{-t_n}(v)}\bigr).
\]
Since $V_{\Phi_{-t_n}(v)}$ is a graph over $R(\Phi_{-t_n}(v))$, the subspace $V_{-t_n}(v)$ is also a graph over $R(v)$, with graph operator
\[
L_{t_n,v}
=
\Psi_{t_n}|_{E^s(\Phi_{-t_n}(v))}
\circ
L_{\Phi_{-t_n}(v)}
\circ
\left(\Psi_{t_n}|_{R(\Phi_{-t_n}(v))}\right)^{-1}.
\]
Moreover,
\begin{align*}
\|L_{t_n,v}\|
\le &
\|\Psi_{t_n}|_{E^s(\Phi_{-t_n}(v))}\|\,
\|L_{\Phi_{-t_n}(v)}\|\,
\|\left(\Psi_{t_n}|_{R(\Phi_{-t_n}(v))}\right)^{-1}\|\\
\le &L_M(\lambda\eta)^{t_n}
\end{align*}

Since $\lambda\eta<1$, $\|L_{t_n,v}\|\to0$ as $n\to \infty.$ Therefore,

\[
F(v)=\lim_{n \to \infty}V_{t_n}(v)=\lim_{n \to \infty} \Psi_{t_n}({R(\Phi_{-t_n}(v)))}\subseteq E^c(v)\oplus E^u(v).
\]

Since this inclusion holds for every $v\in K$ and $\mu(K)>1-\varepsilon$, where $\varepsilon>0$ is arbitrary, it follows that the inclusion holds for $\mu$-almost every $v\in\mathcal O$.\\

\noindent \textbf{Step 2: proof of the inclusion $E^u(v)\subset F(v)$.}

Identifying each fiber of the reduced bundle with
\[
Q_v\simeq v^\perp\oplus v^\perp,
\]
we endow $Q_v$ with the canonical inner product
\[
\langle \langle(\xi_1,\eta_1),(\xi_2,\eta_2)\rangle \rangle=\langle \xi_1,\xi_2\rangle_g+\langle \eta_1,\eta_2\rangle_g,
\]
and the compatible complex structure
\[
\mathfrak J(\xi,\eta)=(-\eta,\xi).
\]

It follows from Lemma 4.1 of \cite{assenza2025magnetic} that the symplectic form satisfies
\[
\omega_\sigma(\zeta_1,\zeta_2)=\langle \langle \zeta_1,\mathfrak J\zeta_2\rangle \rangle.
\]

Consequently, for every Lagrangian subspace $F(v)\subset Q_v$ we obtain
the orthogonal decomposition
\[
Q_v=F(v)\oplus \mathfrak JF(v).
\]

Moreover, since $F(v)$ is Lagrangian, $\mathfrak JF(v)$ is also Lagrangian.

Choose an orthonormal basis
\[
e_1(v),\dots,e_{n-1}(v)
\]
of $F(v)$ and complete it to an orthonormal symplectic basis of $Q_v$ by defining
\[
e_{n+i}(v)=\mathfrak J e_i(v),\qquad i=1,\dots,n-1,
\]
and do the same along the orbit in $Q_{\Phi_t(v)}$.

With respect to
\[
Q_v=F(v)\oplus \mathfrak JF(v),\qquad
Q_{\Phi_t(v)}=F(\Phi_t(v))\oplus \mathfrak JF(\Phi_t(v)),
\]
we have
\[
\Psi_t=
\begin{pmatrix}
A_t(v) & C_t(v)\\
0 & D_t(v)
\end{pmatrix}
\qquad
\text{and}
\qquad
D_t(v)=(A_t(v)^*)^{-1},
\]
because $\Psi_t$ preserves the form symplectic and $(\Psi_t)^{*}\mathfrak J(\Psi_t)=\mathfrak J$. Now, note that as $(\Psi_t)(\Psi_{-t})=I$, we have
\[
\Psi_{-t}=
\begin{pmatrix}
A_t(v)^{-1} & *\\
0 & A_t(v)^*
\end{pmatrix}.
\]

Suppose, by contradiction, that there exists
\[
\xi\in E^u(v)\setminus F(v).
\]
Writing
\[
\xi=z+w,\qquad z\in F(v),\quad w\in \mathfrak JF(v),\quad w\neq 0.
\]

Applying $\Psi_{-t}$ to $\xi=z+w$, we obtain
\[
\Psi_{-t}(\xi) = At^{-1}z+(*\,w)+A_t^*w.
\]
The component of $\Psi_{-t}(\xi)$ in
$\mathcal JF(\Phi_{-t}(v))$ is exactly $A_t^*w$. Therefore
\[
\|\Psi_{-t}(\xi)\|\ge\|A_t^*w\|\qquad \text{ and } \qquad \lim_{t\to +\infty}\frac1t\log\|\Psi_{-t}(\xi)\|\ge\lim_{t\to +\infty}\frac1t
\log\|A_t^*w\|.
\]

Now, all Lyapunov exponents of $A_t$ on $F(v)$ are nonnegative because $F(v)\subset E^u(v)\oplus E^c(v)$. Indeed, $F(v)$ contains the unstable Oseledec space and
possibly part of the central space.

Since $A_t$ and its adjoint $A_t^*$ have the same singular values,
they have the same Lyapunov spectrum. Therefore, because all Lyapunov
exponents of $A_t$ on $F(v)$ are nonegative, the same holds for
$A_t^*$.  Since $w\neq0$, it follows
that
\[
\lim_{t\to +\infty}\frac1t\log\|\Psi_{-t}(\xi)\|\ge\lim_{t\to+\infty}\frac1t\log\|A_t^*w\|\ge 0.
\]

On the other hand, because
\[
\xi\in E^u(v),
\]
all Lyapunov exponents of $\xi$ are strictly positive.

Equivalently, in negative time one has exponential contraction:
There exists $\lambda>0$ such that
\[
\lim_{t\to+\infty}\frac1t\log\|\Psi_{-t}(\xi)\|=-\lambda<0.
\]
This contradicts the previous inequality.

Therefore, no such vector $\xi$ can exist, and we conclude that
\[
E^u(v)\subset F(v).
\]

\noindent \textbf{Step 3:} Finally, the existence of the limit 
\[
\lim_{T\to+\infty}\frac1T\log\left|\det\!\left(\Psi_T\big|_{E^{\sigma,-}(v)}\right)\right|
\]
follows immediately from Oseledec's theorem.

\end{proof}

We can now combine the preceding proposition with the properties of the unstable Green bundle established earlier.

\begin{proof}[Proof of Theorem \ref{teorema expoente}]

Fix a point $v\in\mathcal O$. Since $E^{\sigma,-}(v)$ is the graph of the
symmetric operator $\mathcal S_v^-:v^\perp\to v^\perp$, the linear map
\[
\Pi_v:=d\pi_v\big|_{E^{\sigma,-}(v)}:E^{\sigma,-}(v)\longrightarrow v^\perp
\]
is an isomorphism.
Note that
\begin{equation}\label{conjugação Y e psi}
Y^-_v(t)=\Pi_{\Phi_t(v)}\circ\Psi_t\big|_{E^{\sigma,-}(v)}\circ
\Pi_v^{-1}:v^\perp\longrightarrow v^\perp. 
\end{equation}

For an invertible differentiable matrix $A(t)$, we have
\[
\frac{d}{dt}\log|\det A(t)|=\tr\!\bigl(\dot A(t)A(t)^{-1}\bigr).
\]
Applying this to $Y^-_v(t)$, we obtain
\[
\frac{d}{dt}\log|\det Y^-_v(t)|=\tr\!\bigl(\dot Y^-_v(t)Y^-_v(t)^{-1}\bigr)=\tr\bigl(\mathcal U^-_{\Phi_t(v)}(0)\bigr).
\]
Therefore,
\[
\frac1T\log|\det Y^-_v(T)|=\frac1T\int_0^T \tr\bigl(\mathcal U^-_{\Phi_t(v)}(0)\bigr)\,dt.
\]

Follows from \eqref{conjugação Y e psi} that
\[
\Psi_T\big|_{E^{\sigma,-}(v)}=\Pi_{\Phi_T(v)}^{-1}\circ Y^-v(T)\circ \Pi_v.
\]
Then
\[
\log\left|\det\!\left(\Psi_T\big|_{E^{\sigma,-}(v)}\right)\right|=
\log|\det \Pi^{-1}_{\Phi_T(v)}|+|+ \log|\det Y^-_v(T)| +\log|\det \Pi_v|
\]
It is immediate that $\Pi_v$ is uniformly bounded. Moreover, since $\mathcal U^-$ is uniformly bounded on $\Sigma_s$ by Proposition~\ref{Riccati limitado}, it follows that $\Pi_v^{-1}$ is also uniformly bounded. Then, by Proposition \ref{osedelecs}
\[
\chi^\sigma(v)=\lim_{T\to+\infty}\frac1T\log\left|\det\!\left(\Psi_T\big|_{E^{\sigma,-}(v)}\right)\right|=\lim_{T\to+\infty}\frac1T
\int_0^T \tr\bigl(\mathcal U^-_{\Phi_t(v)}(0)\bigr)\,dt.
\]
for $\mu$-almost every $v\in\Sigma_s.$
\end{proof}

\begin{proof}[Proof of Corollary \ref{segundo corolario}] Analogously to the proof of Theorem~A in \cite{assenza2025magnetic}, one sees that

\begin{equation}\label{integral tr and Ric}
\int_{\Sigma_s}\tr\bigl((\mathcal U^-_v(0))^2\bigr)\,d\mu=-(n-1)
\int_{\Sigma_s}\Ric^\Omega_s(v)\,d\mu.
\end{equation}

Let $\lambda_1(v),\ldots,\lambda_{n-1}(v)$ denote the eigenvalues of
the symmetric operator $\mathcal U^-_v(0)$. Using Cauchy--Schwarz inequality,
\[
\bigl(\tr\mathcal U^-_v(0)\bigr)^2=\left(\sum_{i=1}^{n-1}\lambda_i(v)\right)^2\le(n-1)\sum_{i=1}^{n-1}\lambda_i(v)^2=(n-1)\tr\bigl((\mathcal U^-_v(0))^2\bigr).
\]

Then

\begin{equation}\label{integral U e U ao quadrado}
(\mathfrak h_\mu^\sigma)^2=\left(\int_{\Sigma_s}\tr(\mathcal U^-_v(0))\,d\mu\right)^2\le\int_{\Sigma_s}(\tr\mathcal U^-_v(0))^2\,d\mu
\le(n-1)\int_{\Sigma_s}\tr\bigl((\mathcal U^-_v(0))^2\bigr)\,d\mu.
\end{equation}

Therefore by \eqref{integral tr and Ric}
\[
(\mathfrak h_\mu^\sigma)^2\le-(n-1)^2\int_{\Sigma_s}\Ric^\Omega_s(v)\,d\mu.
\]

And
\[
\mathfrak h_\mu^\sigma\le(n-1)\left(-\int_{\Sigma_s}\Ric^\Omega_s(v)\,d\mu\right)^{1/2}.
\]

Assume now that equality holds. By \eqref{integral U e U ao quadrado}, we have
\[
(\tr\mathcal U^-_v(0))^2=(n-1)\tr\bigl((\mathcal U^-_v(0))^2\bigr)
\]
and the equality case in the Cauchy--Schwarz inequality implies that all eigenvalues of $\mathcal U^-_v(0)$ coincide for
$\mu$-almost every $v$. Hence
\[
\mathcal U^-_v(t)=\lambda_v(t)I.
\]
Substituting into the magnetic Riccati equation yields
\[
\dot{\lambda}_v(t)I+\lambda_v(t)^2I+(M_s^\Omega)_{\Phi_t(v)}=0.
\]
Therefore

\[
(M_s^\Omega)_v=-\bigl(\dot{\lambda}_v(0)+\lambda_v(0)^2\bigr)I.
\]
Hence $(M_s^\Omega)_v$ is a scalar multiple of the identity for $\mu$-almost every $v\in\Sigma_s$.

\end{proof}

We now introduce the notion of absence of magnetic focal points, which will serve as the geometric assumption ensuring that the unstable Riccati operator is positive semidefinite.

\begin{definition}
We say that the magnetic system $(M,g,\sigma)$ has \textit{no focal points
along the magnetic trajectory} $\gamma$ if, for every nontrivial
normal magnetic Jacobi field $J$ along $\gamma$ with $J(0)=0$, the
function
\[
t\longmapsto \|J^\perp(t)\|^2
\]
is strictly increasing on $(0,\infty)$.

Equivalently,
\[
\frac{d}{dt}\|J^\perp(t)\|^2>0\qquad \text{for every }t>0,
\]
for every nontrivial normal magnetic Jacobi field $J$ with $J(0)=0$. 
\end{definition}

It is clear that magnetic systems without focal points have no conjugate points. Moreover, the argument in the proof of Proposition~3.2 of \cite{assenza2025magnetic} shows that magnetic systems with non-positive magnetic sectional curvature are without focal points.

Geometrically, the absence of focal points strengthens the absence of conjugate points by requiring a monotonicity property for normal magnetic Jacobi fields. That is, if a normal magnetic Jacobi field is zero at some point, then its norm increases strictly along the trajectory. This additional geometric control plays a fundamental role in the sequence. As in the classical Riemannian setting, it yields positivity properties for the Green bundles.

\begin{lemma}\label{lema U^{+} positivo}
Assume that $(M,g,\sigma)$ has no focal points on $\Sigma_s$. Then the unstable Riccati
operator $\mathcal U^{-}_v$ is positive semidefinite for every
$v\in\Sigma_s$.
\end{lemma}

\begin{proof}
Fix $v\in\Sigma_s$, and let $Y_{v,T}$ be the matrix solution of \eqref{Jacobi eq} satisfying $Y_T(0)=Id
$ and $Y_T(T)=0$.

Since $(M,g,\sigma)$ has no focal points, every nontrivial normal magnetic
Jacobi field $J$ with $J^\perp(0)=0$ satisfies that
\[
t\longmapsto \|J^\perp(t)\|^2
\]
is strictly increasing on $(0,\infty)$. Applying this in the
fields
\[
J_x^\perp(t):=P_{-t}(Y_T(t)x),\qquad x\in v^\perp\setminus\{0\},
\]
we get
\[
\frac{d}{dt}\|Y_T(t)x\|^2>0\qquad\text{for all }t>T.
\]
In particular, no focal points imply no conjugate points, so $Y_{T}(t)$ is
invertible for every $t>T$.

Define
\[
U_T(t):=\dot Y_T(t)(Y_T(t))^{-1}, \,\, \mbox{for } t>T.
\]
If $\eta=Y_T(t)x$, then
\[
\langle U_T(t)\eta,\eta\rangle=\langle \dot Y_T(t)x,Y_T(t)x\rangle=\frac12\frac{d}{dt}\|Y_T(t)x\|^2>0,\,\,\,\forall t>T. 
\]

Note that if $T\to -\infty$ then $U_T(0)\to \dot Y_v^-(0)=\mathcal U_v^{-}(0)$ then $\langle \mathcal U_v^{-}(t)x,x\rangle\ge0$ for all $x\in v^\perp.$
\end{proof}

\begin{proof}[Proof of Theorem \ref{teorema rigidez}]
Suppose, by contradiction, that
\[
h_{\mu_L}(\Phi)=0.
\]

Since $\mu_L$ is a Lebesgue measure on $\Sigma_s$, by Corollary \ref{Corolário Lebesgue} 
\[
h_{\mu_L}(\Phi) = \int_{\Sigma_s}\tr\bigl(\mathcal U^-_v(0)\bigr)\,d\mu_L(v).
\]
Hence
\[
\int_{\Sigma_s}\tr\bigl(\mathcal U^-_v(0)\bigr)\,d\mu_L(v)=0.
\]

Because the magnetic system has no focal points, Lemma \ref{lema U^{+} positivo} implies that
$
\mathcal U^-_v(0)
$ is positive semidefinite
for every $v\in\Sigma_s$. In particular,
\[
\tr\bigl(\mathcal U^-_v(0)\bigr)\ge 0.
\]

Since $\mathcal U^-_v(0)$ is measurable, we obtain that
\[
\tr\bigl(\mathcal U^-_v(0)\bigr)=0\qquad\text{for $\mu_L$-almost every }v\in\Sigma_s.
\]
Again, since $\mathcal U^-_v(0)$ is symmetric and  positive semidefinite, the vanishing of its trace implies that
\[
\mathcal U^-_v(0)=0\qquad\text{for $\mu_L$-almost every} \,v\in\Sigma_s.
\]

Therefore, from the Riccati equation, we conclude
\[
M_s^\Omega\equiv 0 \qquad\text{for $\mu_L$-almost every }v\in\Sigma_s
\]

It follows that
\[
M_s^\Omega\equiv 0\qquad\text{on }\Sigma_s.
\]
This contradiction concludes the proof.
\end{proof}

\bibliographystyle{alpha}
\bibliography{main}
 
\end{document}